\documentclass[12pt, reqno]{amsart}
\usepackage{amsmath, amsthm, amscd, amsfonts, amssymb, graphicx, color}
\usepackage[bookmarksnumbered, colorlinks, plainpages]{hyperref}
\input{mathrsfs.sty}

\newtheorem{theorem}{Theorem}[section]
\newtheorem{lemma}[theorem]{Lemma}

\newtheorem{corollary}[theorem]{Corollary}
\theoremstyle{definition}

\theoremstyle{remark}

\numberwithin{equation}{section}

\begin{document}

\title[Characterization of a generalized triangle inequality]{Characterization of a generalized triangle inequality in normed spaces}
\dedicatory{{\rm {\bf Farzad Dadipour} \\Department of Pure Mathematics, Ferdowsi University of
Mashhad, P. O. Box 1159, Mashhad 91775, Iran\\dadipoor@yahoo.com\\
\vspace{0.25cm}
{\bf Mohammad Sal Moslehian}\\Department of Pure Mathematics, Center of Excellence in
Analysis on Algebraic Structures (CEAAS)\\ Ferdowsi University of
Mashhad, P. O. Box 1159, Mashhad 91775, Iran\\moslehian@ferdowsi.um.ac.ir\\
\vspace{0.25cm}
{\bf John M. Rassias}\\Pedagogical Department, National and Capodistrian University of
Athens, Section of Mathematics and Informatics, 4, Agamemnonos str.,
Aghia Paraskevi, Attikis 15342, Athens, Greece\\jrassias@primedu.uoa.gr\\
\vspace{0.25cm}
{\bf Sin-Ei Takahasi}\\ Yamagata University (Professor Emeritus), Natsumidai 3-8-16-502,Funabashi, Chiba, 273-0866, Japan \\sin$_-$ei1@yahoo.co.jp}}



\begin{abstract}
For a normed linear space $(X,\|\cdot\|)$ and $p>0$ we characterize all
$n$-tuples $(\mu_1,\cdots,\mu_n)\in\mathbb{R}^{n}$ for which the
generalized triangle inequality of the second type
$$\|x_1+\cdots+x_n\|^p\leq\frac{\|x_1\|^p}{\mu_1}+\cdots+\frac{\|x_n\|^p}{\mu_n}$$ holds for any $x_1,\cdots,x_n\in
X$. We also characterize $(\mu_1,\cdots,\mu_n)\in\mathbb{R}^{n}$
for which the reverse of the inequality above holds.
\end{abstract} \maketitle

%------------------------------------------------------------------------------%

\section{Introduction}

The triangle inequality is one of the most significant inequalities
in mathematics. It has many interesting generalizations, refinements
and reverses, which have been obtained over the years, see \cite{AM,
Drag, KMM, MAL, MSKT, Pec-Raj} and references therein. The generalized triangle inequalities are useful to study the geometrical structure of normed spaces, see e.g. \cite{Kato, Hsu, HL}. In this direction
some results have been based on the triangle inequality of the
second type
\begin{eqnarray*}
\|x+y\|^2\leq2(\|x\|^2+\|y\|^2)
\end{eqnarray*}
in a normed linear space, see \cite{Belb, Saitoh, Takahasi-Ras-Sai-Takahashi} for more information about this inequality.\\
In framework of Hilbert spaces the Euler-Lagrange type identity (see \cite{Ras})
$$\frac{\|x\|^2}{\mu}+\frac{\|y\|^2}{\nu}-\frac{\|ax+by\|^2}{\lambda}=\frac{\|\nu
bx-\mu ay\|^2}{\lambda\mu\nu}\ \ \ (\lambda=\mu a^2+\nu b^2)$$
follows the more general triangle inequality of the second type (see
\cite{Takahasi-Ras-Sai-Takahashi})
\begin{eqnarray*}
\frac{\|ax+by\|^2}{\lambda}\leq\frac{\|x\|^2}{\mu}+\frac{\|y\|^2}{\nu}\,,
\end{eqnarray*}
where $\lambda=\mu a^2+\nu b^2$ and $\lambda\mu\nu>0$. Also Takahasi
et al. \cite{Takahasi-Ras-Sai-Takahashi} investigate the inequality
\begin{eqnarray}\label{mos}
\frac{\|ax+by\|^p}{\lambda}\leq\frac{\|x\|^p}{\mu}+\frac{\|y\|^p}{\nu}
\end{eqnarray}
for $p\geq1$. We should notice that one can assume that $\lambda=\pm 1$ by dividing the both sides of \eqref{mos} by $|\lambda|$ as well as  $a\neq 0, b \neq 0$ since if, e.g. $a=0$, then $\mu$ can be suitably chosen arbitrary.  By replacing $x$ and $y$ by $\frac{x}{a}$ and $\frac{y}{b}$ and changing  $\mu$ and $\nu$ accordingly, inequality \eqref{mos} turns into $\frac{\|x+y\|^p}{\pm 1}\leq\frac{\|x\|^p}{\mu}+\frac{\|y\|^p}{\nu}$.

In this paper, for a normed linear space $(X,\|\cdot\|)$ and $p>0$ we
characterize all $n$-tuples $(\mu_1,\cdots,\mu_n)\in\mathbb{R}^{n}$
for which the generalized triangle inequality of the second type
$$\|x_1+\cdots+x_n\|^p\leq\frac{\|x_1\|^p}{\mu_1}+\cdots+\frac{\|x_n\|^p}{\mu_n}$$ holds for any $x_1,\cdots,x_n\in
X$. We also characterize $(\mu_1,\cdots,\mu_n)\in\mathbb{R}^{n}$
for which the reverse of the inequality above holds.

\section{Main results}

We need the following three lemmas which generalize some results due
to Takagi et al. \cite[Theorem 2.1, Lemma 2.2 and Theorem
2.3]{Takagi-Miu-Hay-Takahasi}. Let us first recall the concept of an
envelope.

An envelope of a family of surfaces is a surface that is
tangent to each member of the family at some point. Let
$\mathscr{L}$ be an $(n-1)$-parameters family of surfaces in
$\mathbb{R}^n$ given by $F(a_1,\cdots,a_n;s_1,\cdots,s_{n-1})=0$
depending on real parameters $s_1,\cdots,s_{n-1}$ and variables
$a_1,\cdots,a_n$. The envelope of $\mathscr{L}$ is the set of points
$(a_1,\cdots,a_n)\in\mathbb{R}^n$ for which the following equations
hold:
\begin{eqnarray*}
&F(a_1,\cdots,a_n;s_1,\cdots,s_{n-1})=0\,,&\\
&\frac{\partial F}{\partial
s_i}(a_1,\cdots,a_n;s_1,\cdots,s_{n-1})=0&\ \ \ (1\leq i\leq n-1)\,.
\end{eqnarray*}

%Lemma 2.1%

\begin{lemma}\label{lemma1}
Suppose that $p>1$, $S=\{(s_1,\cdots,s_n)\colon s_1,\cdots,s_n\geq0,
\sum_{i=1}^{n}s_i=1\}$ and
$h_p(a_1,\cdots,a_{n-1})=\Big(1-\sum_{i=1}^{n-1}a_i^{\frac{1}{1-p}}\Big)^{1-p}$ for $a_1,\cdots,a_{n-1}>0$ with $\sum_{i=1}^{n-1}a_i^{\frac{1}{1-p}}<1$.
For each $(s_1,\cdots,s_n)\in S$, let\\
\centerline{$L_p(s_1,\cdots,s_n):=\{(a_1,\cdots,a_n)\colon
a_1s_1^p+\cdots+a_ns_n^p=1\}\ ,$}
\centerline{$\Delta_p(s_1,\cdots,s_n):=\{(a_1,\cdots,a_n)\colon
a_1s_1^p+\cdots+a_ns_n^p\geq1\}\,.$}
Then the following assertions hold:\\
{\rm (i)} If $\mathscr{L}=\{L_p(s_1,\cdots,s_n)\colon(s_1,\cdots,s_n)\in S\}$, then the envelope of $\mathscr{L}$ is given by $a_n=h_p(a_1,\cdots,a_{n-1})$;\\
{\rm (ii)} $\displaystyle{\bigcap_{(s_1,\cdots,s_n)\in
S}}\Delta_p(s_1,\cdots,s_n)=\{(a_1,\cdots,a_n)\colon a_n\geq
h_p(a_1,\cdots,a_{n-1})\}\,.$
\end{lemma}

\begin{proof}
(i) Putting\\
\centerline{$F(a_1,\cdots,a_n;s_1,\cdots,s_{n-1})=a_1s_1^p+\cdots+a_{n-1}s_{n-1}^p+a_n(1-(s_1+\cdots+s_{n-1}))^p-1$,}
we can consider $\mathscr{L}$ as a family of $(n-1)$-parameters of surfaces as follows:\\
\begin{eqnarray*}
\mathscr{L}&=&\{L_p(s_1,\cdots,s_n)\colon(s_1,\cdots,s_n)\in S\}\\
&=&\Big\{\{(a_1,\cdots,a_n)\colon a_1s_1^p+\cdots+a_ns_n^p=1\}\colon s_1,\cdots,s_n\geq0, \sum_{i=1}^{n}s_i=1\Big\}\\
&=&\Big\{\{(a_1,\cdots,a_n)\colon a_1s_1^p+\cdots+a_{n-1}s_{n-1}^p+a_n(1-(s_1+\cdots+s_{n-1}))^p=1\}\colon\\
&&s_1,\cdots,s_{n-1}\geq 0, \sum_{i=1}^{n-1}s_i\leq1\Big\}\\
&=&\Big\{\{(a_1,\cdots,a_n)\colon
F(a_1,\cdots,a_n;s_1,\cdots,s_{n-1})=0\}\colon
s_1,\cdots,s_{n-1}\geq0, \sum_{i=1}^{n-1}s_i\leq1\Big\}\,.
\end{eqnarray*}
The envelope of $\mathscr{L}$ is given by the solutions of the
following simultaneous equations:
\begin{eqnarray}\label{L1}
&F(a_1,\cdots,a_n;s_1,\cdots,s_{n-1})=0\,,&\nonumber\\
&\frac{\partial F}{\partial
s_i}(a_1,\cdots,a_n;s_1,\cdots,s_{n-1})=0&\ \ \ (1\leq i\leq n-1)\,.
\label{L2}
\end{eqnarray}
Equations \eqref{L2} yield
\begin{eqnarray}\label{L2.5}
pa_is_i^{p-1}-pa_n(1-(s_1+\cdots+s_{n-1}))^{p-1}=0\ \ \ (1\leq i\leq n-1)
\end{eqnarray}
It must be $s_i\neq 0$ for all $i=1,\cdots,n-1$. Indeed, if $s_j=0$ for some $1 \leq j\leq n-1$, then
$a_1s_1=\cdots=a_{n-1}s_{n-1}=a_n\big(1-\sum_{i=1}^{n-1}s_i\big)^{p-1}=0$ by \eqref{L2.5}. So that $F(a_1,\cdots,a_n;s_1,\cdots,s_{n-1})=-1$, which contradicts  \eqref{L1}. Moreover, by the same way, it must be $\sum_{i=1}^{n-1}s_i< 1$.\\
From \eqref{L2.5} we get
\begin{eqnarray}\label{L3}
a_i=\frac{a_n(1-(s_1+\cdots+s_{n-1}))^{p-1}}{s_i^{p-1}}\qquad (1\leq
i\leq n-1)\,.
\end{eqnarray}
Using \eqref{L3}, equation \eqref{L1} turns into\\
\begin{eqnarray*}
a_ns_1(1-(s_1+\cdots+s_{n-1}))^{p-1}+\cdots+a_ns_{n-1}(1-(s_1+\cdots+s_{n-1}))^{p-1}\\
+a_n(1-(s_1+\cdots+s_{n-1}))^{p}-1=0\,,
\end{eqnarray*}
or equivalently we get
\begin{eqnarray}\label{L4}
a_n=\frac{1}{(1-(s_1+\cdots+s_{n-1}))^{p-1}}\,.
\end{eqnarray}
From \eqref{L3} and \eqref{L4} we get
\begin{eqnarray}\label{L5}
a_i=\frac{1}{s_i^{p-1}}\ \ \ (1\leq i\leq n-1)\,.
\end{eqnarray}
We note that $s_1,\cdots,s_{n-1}>0$ and $\sum_{i=1}^{n-1}s_i<1$ if
and only if $a_1,\cdots,a_{n-1}>0$ and
$\sum_{i=1}^{n-1}a_i^{\frac{1}{1-p}}<1$. Now we remove the
parameters $s_1,\cdots,s_{n-1}$ in equations \eqref{L4} and
\eqref{L5} to get
\begin{eqnarray*}
a_n&=&\frac{1}{\left(1-\left(\left(\frac{1}{a_1}\right)^{\frac{1}{p-1}}+\cdots+\left(\frac{1}{a_{n-1}}\right)^{\frac{1}{p-1}}\right)\right)^{p-1}}
=\left(1-\left(a_1^{\frac{1}{1-p}}+\cdots+a_{n-1}^{\frac{1}{1-p}}\right)\right)^{1-p}\\
&=&\left(1-\sum_{i=1}^{n-1}a_i^{\frac{1}{1-p}}\right)^{1-p}=h_p(a_1,\cdots,a_{n-1})\,.
\end{eqnarray*}
(ii) It is not difficult to check that the function $h_p$ is
strictly convex and thus the domain $\{(a_1,\cdots,a_n)\colon a_n\geq
h_p(a_1,\cdots,a_{n-1})\}$ is a strictly convex set in the Euclidean
space $\mathbb{R}^n$. Now the result follows from part (i).
\end{proof}

%Lemma 2.2%

\begin{lemma}\label{lemma2}
Let $p>1$ and $\Omega\subseteq \big\{(s_1,\cdots,s_n)\colon
s_1,\cdots,s_n\geq0,\ \sum_{i=1}^{n}s_i\geq1\big\}$. Let
$D_p(\Omega):=\{(a_1,\cdots,a_n)\colon a_1,\cdots,a_n\geq 0,
 a_1s_1^p+\cdots+a_ns_n^p\geq 1\ for\ all\ (s_1,\cdots,s_n)\in\Omega\}$.
Then the following assertions hold:\\
{\rm (i)} $\{(a_1,\cdots,a_n)\colon a_n\geq
h_p(a_1,\cdots,a_{n-1})\}\subseteq D_p(\Omega)\,;$\\
{\rm (ii)} If $S\subseteq\overline{\Omega}$ where
$\overline{\Omega}$ is the usual closure of $\Omega$, then
$D_p(\Omega)=\{(a_1,\cdots,a_n)\colon a_n\geq
h_p(a_1,\cdots,a_{n-1})\}\,.$
\end{lemma}

\begin{proof}
(i) By Lemma \ref{lemma1} (ii) it is sufficient to show that\\
\centerline{$\displaystyle{\bigcap_{(s_1,\cdots,s_n)\in
S}}\Delta_p(s_1,\cdots,s_n)\subseteq D_p(\Omega)$.} Let
$(a_1,\cdots,a_n)\in \displaystyle{\bigcap_{(s_1,\cdots,s_n)\in
S}}\Delta_p(s_1,\cdots,s_n)$ and $(t_1,\cdots,t_n)\in\Omega$ be
arbitrary. We can easily find $(r_1,\cdots,r_n)\in S$ such that
$r_i\leq t_i \ (1\leq i\leq n)$ whence $a_1r_1^p+\cdots+a_nr_n^p\leq
a_1t_1^p+\cdots+a_nt_n^p$.\\ Also from
$(a_1,\cdots,a_n)\in \displaystyle{\bigcap_{(s_1,\cdots,s_n)\in
S}}\Delta_p(s_1,\cdots,s_n)$ we get $a_1r_1^p+\cdots+a_nr_n^p\geq1$.\\
Thus $a_1t_1^p+\cdots+a_nt_n^p\geq1$ whence $(a_1,\cdots,a_n)\in
D_p(\Omega)$.\\
(ii) By (i) and Lemma \ref{lemma1} (ii) it is sufficient to show
that
$$D_p(\Omega)\subseteq \displaystyle{\bigcap_{(s_1,\cdots,s_n)\in
S}}\Delta_p(s_1,\cdots,s_n)\,.$$ Let $(a_1,\cdots,a_n)\in
D_p(\Omega)$ and $(s_1,\cdots,s_n)\in S$. There exists
a sequence $\{(t_{1,m},\cdots,t_{n,m})\}_{m=1}^{\infty}$ in $\Omega$
satisfying $(t_{1,m},\cdots,t_{n,m})\longrightarrow(s_1,\cdots,s_n)$
as $m\longrightarrow\infty$, or equivalently $t_{i,m}\longrightarrow
s_i$ as $m\longrightarrow\infty\,\,(1\leq i\leq n)$, since
$S\subseteq\overline{\Omega}$.\\
Now $(a_1,\cdots,a_n)\in D_p(\Omega)$ implies that
$a_1t_{1,m}^p+\cdots+a_nt_{n,m}^p\geq1\ (m\in\mathbb{N})$. Getting
limit as $m\longrightarrow\infty$, it follows that
$a_1s_1^p+\cdots+a_ns_n^p\geq1$.\\
Hence $(a_1,\cdots,a_n)\in\displaystyle{\bigcap_{(s_1,\cdots,s_n)\in
S}}\Delta_p(s_1,\cdots,s_n)$.
\end{proof}

Next, we identify $D_p(\Omega)$ when $0<p\leq1$.

%Lemma 2.3%

\begin{lemma}\label{lemma3}
Let $0<p\leq1,\ \Omega\subseteq\{(s_1,\cdots,s_n)\colon
s_1,\cdots,s_n\geq0,\ \sum_{i=1}^{n}s_i\geq1\}$ and $D_p(\Omega)$ be
as in Lemma \ref{lemma2}. Then the following assertions hold:\\
{\rm (i)} $\{(a_1,\cdots,a_n)\colon a_1\geq1,\cdots,a_n\geq1\}\subseteq D_p(\Omega)$;\\
{\rm (ii)} If $\{e_1,\cdots,e_n\}\subseteq\overline{\Omega}$ where
$\{e_1,\cdots,e_n\}$ is the standard basis of $\mathbb{R}^n$, then\\
$D_p(\Omega)=\{(a_1,\cdots,a_n)\colon a_1\geq1,\cdots,a_n\geq1\}.$
\end{lemma}

\begin{proof}
(i) Let $a_1\geq1,\cdots,a_n\geq1$ and $(s_1,\cdots,s_n)\in\Omega$
be arbitrary. First we show that $s_1^p+\cdots+s_n^p\geq1$.\\
Case 1. Let $s_i\leq1$ for all $i=1,\cdots,n$.\\
It follows that $s_i\leq s_i^p$ for all $i=1,\cdots,n$, whence
$s_1+\cdots+s_n\leq s_1^p+\cdots+s_n^p$. Also from
$(s_1,\cdots,s_n)\in\Omega$ we get $s_1+\cdots+s_n\geq1$. Thus
$s_1^p+\cdots+s_n^p\geq1$.\\
Case 2. Let $s_j>1$ for some $1 \leq j\leq n$.\\
We can easily obtain that $s_1^p+\cdots+s_n^p\geq s_j^p>1$.

Now we observe that $a_1s_1^p+\cdots+a_ns_n^p\geq
s_1^p+\cdots+s_n^p\geq1$. Hence
$(a_1,\cdots,a_n)\in D_p(\Omega)$.\\
(ii) By (i) it suffices to show that
$D_p(\Omega)\subseteq\{(a_1,\cdots,a_n)\colon
a_1\geq1,\cdots,a_n\geq1\}\,.$ Let $(a_1,\cdots,a_n)\in D_p(\Omega)$
and $k\in\{1,\cdots,n\}$. There exists a sequence
$\{(s_{1,m},\cdots,s_{n,m})\}_{m=1}^{\infty}$ in $\Omega$ satisfying
$(s_{1,m},\cdots,s_{n,m})\longrightarrow e_k$ as
$m\longrightarrow\infty$, or equivalently $s_{k,m}\longrightarrow1$
and $s_{l,m}\longrightarrow0$ as $m\longrightarrow\infty$
$(l\in\{1,\cdots,n\}\setminus\{k\})$, since
$\{e_1,\cdots,e_n\}\subseteq\overline{\Omega}$. Also from
$(a_1,\cdots,a_n)\in D_p(\Omega)$ we get
$a_1s_{1,m}^p+\cdots+a_ns_{n,m}^p\geq1\ \ (m\in\mathbb{N})$. Taking
limit as $m\longrightarrow\infty$, it follows that $a_k\geq1$.

\end{proof}

%------------------------------------------------------------------------------%

Let $(X,\|\cdot\|)$ be a normed space and $p>0$. Our main aim is to a characterize all $n$-tuples
$(\mu_1,\cdots,\mu_n)\in\mathbb{R}^{n}$ satisfying
\begin{eqnarray}\label{L6}
\|x_1+\cdots+x_n\|^p\leq\frac{\|x_1\|^p}{\mu_1}+\cdots+\frac{\|x_n\|^p}{\mu_n}\
\ (x_1,\cdots,x_n\in X)
\end{eqnarray}
or its reverse
\begin{eqnarray}\label{L6.5}
\|x_1+\cdots+x_n\|^p\geq\frac{\|x_1\|^p}{\mu_1}+\cdots+\frac{\|x_n\|^p}{\mu_n}\
\ (x_1,\cdots,x_n\in X)\,.
\end{eqnarray}
We put
$$F(p)=\left\{(\mu_1,\cdots,\mu_n)\in\mathbb{R}^{n}\colon\ \ \left\|\sum_{i=1}^nx_i\right\|^p
\leq\sum_{i=1}^{n}\frac{\|x_i\|^p}{\mu_i}\ {\rm for\ all}\
x_1,\cdots,x_n\in
X\right\}$$and
$$G(p)=\left\{(\mu_1,\cdots,\mu_n)\in\mathbb{R}^{n}\colon\
\ \left\|\sum_{i=1}^nx_i\right\|^p \geq\sum_{i=1}^{n}\frac{\|x_i\|^p}{\mu_i}\
{\rm for\ all}\ x_1,\cdots,x_n\in X\right\}\,.$$ Also for each
$k=0,1,\cdots,n$ we correspond $F(p\ ;k)$ ($G(p\ ;k)$, resp.) as the
subset of $F(p)$ ($G(p)$, resp.) consisting of all $n$-tuples
$(\mu_1,\cdots,\mu_n)\in\mathbb{R}^n$ for which inequality
\eqref{L6} (\eqref{L6.5}, resp.) holds and exactly $k$ numbers of
$\mu_1,\cdots,\mu_n$ are negative. We note that
\begin{eqnarray}\label{L14}
F(p)=\bigcup_{k=0}^nF(p\ ;k)
\end{eqnarray}
and
\begin{eqnarray}\label{L14.5}
G(p)=\bigcup_{k=0}^nG(p\ ;k)\,.
\end{eqnarray}

In the next two theorems we characterize $F(p)$. First we consider
the case where $p>1$.

%Theorem 2.4%

\begin{theorem}\label{T1}
Let $(X,\|\cdot\|)$ be a normed space and $p>1$. Then the following assertions
hold:\\
{\rm(i)} $F(p\
;0)=\left\{(\mu_1,\cdots,\mu_n)\colon\mu_1,\cdots,\mu_n>0\
and\ \sum_{i=1}^n\mu_i^{\frac{1}{p-1}}\leq 1\right\}\,;$\\
{\rm(ii)} $F(p\ ;k)=\varnothing$, for all $k=1,\cdots,n$;\\
{\rm(iii)} $F(p)=F(p\ ;0)\,.$
\end{theorem}

\begin{proof}
Let $\mu_1,\cdots,\mu_n$ be arbitrary positive numbers for which
\begin{eqnarray*}
\|x_1+\cdots+x_n\|^p\leq\frac{\|x_1\|^p}{\mu_1}+\cdots+\frac{\|x_n\|^p}{\mu_n}\
\ (x_1,\cdots,x_n\in X)\,.
\end{eqnarray*}
Putting $x_i=\mu_i^\frac{1}{p-1}x$ (for some $x\neq0$ and for all
$i=1,\cdots,n$) we get
\begin{eqnarray*}
\left\|\mu_1^\frac{1}{p-1}x+\cdots+\mu_n^\frac{1}{p-1}x\right\|^p\leq\frac{\left\|\mu_1^\frac{1}{p-1}x\right\|^p}{\mu_1}+\cdots+\frac{\left\|\mu_n^\frac{1}{p-1}x\right\|^p}{\mu_n}\,,
\end{eqnarray*}
or equivalently we obtain
$\mu_1^\frac{1}{p-1}+\cdots+\mu_n^\frac{1}{p-1}\leq1$ because of
$p>1$.\\
Conversely, if $\mu_1,\cdots,\mu_n>0$ and
$\mu_1^\frac{1}{p-1}+\cdots+\mu_n^\frac{1}{p-1}\leq1$, then the desired inequality is deduced from
following inequalities:
\begin{eqnarray*}
\|x_1+\cdots+x_n\|^p&\leq&(\|x_1\|+\cdots+\|x_n\|)^p\\
&=&\left(\mu_1^\frac{1}{p}\left\|\frac{x_1}{\mu_1^\frac{1}{p}}\right\|+\cdots+\mu_n^\frac{1}{p}\left\|\frac{x_n}{\mu_n^\frac{1}{p}}\right\|\right)^p\\
&\leq&\left(\mu_1^\frac{1}{p-1}+\cdots+\mu_n^\frac{1}{p-1}\right)^{p-1}\left(\frac{\|x_1\|^p}{\mu_1}+\cdots+\frac{\|x_n\|^p}{\mu_n}\right)\\
&\leq&\frac{\|x_1\|^p}{\mu_1}+\cdots+\frac{\|x_n\|^p}{\mu_n}\,.
\end{eqnarray*}
We note that the second inequality follows from the well-known
H\"older inequality.\\
(ii) Let $(\mu_1,\cdots,\mu_n)\in F(p\ ;k)$ for
some $k=1,\cdots,n$. There exists $1\leq j\leq n$ such that
$\mu_j<0$ and inequality \eqref{L6} holds.\\
Putting $x_i=0\ \left(i\in\{1,\cdots,n\}\setminus\{j\}\right)$ and
$x_j\neq0$ in inequality \eqref{L6} we obtain
$\|x_j\|^p\leq\frac{\|x_j\|^p}{\mu_j}$. This is a
contradiction since $\mu_j<0$.\\
(iii) It follows from (i), (ii) and \eqref{L14}.
\end{proof}

%Theorem 2.5%

\begin{theorem}\label{T3}
Let $(X,\|\cdot\|)$ be a normed space and $0<p\leq1$. Then the following
assertions are valid:\\ {\rm(i)} $F(p\ ;0)=(0,1]\times\cdots\times(0,1]$;\\
{\rm(ii)} $F(p\ ;k)=\varnothing$, for all $k=1,\cdots,n$;\\
{\rm(iii)} $F(p)=F(p\ ;0)\,.$
\end{theorem}

\begin{proof}
(i) Let $\mu_1,\cdots,\mu_n$ be arbitrary positive numbers. We
observe that inequality \eqref{L6} holds if and only if
\begin{eqnarray}\label{I1}
\sum_{i=1}^n\frac{\|x_i\|^p}{\mu_i\|x_1+\cdots+x_n\|^p}\geq1
\end{eqnarray}
for all $x_1,\cdots,x_n\in X$ for which $\sum_{i=1}^nx_i\neq0$.
Putting
$$\Omega=\left\{\left(\frac{\|x_1\|}{\|\sum_{i=1}^nx_i\|},\cdots,\frac{\|x_n\|}{\|\sum_{i=1}^nx_i\|}\right)\colon
x_1,\cdots,x_n\in X,\ \sum_{i=1}^nx_i\neq0 \right\}$$ we get
$\Omega\subseteq\big\{(s_1,\cdots,s_n)\colon s_1,\cdots,s_n\geq0,\
\sum_{i=1}^ns_i\geq1\big\}$,
$\{e_1,\cdots,e_n\}\subseteq\Omega\subseteq\overline{\Omega}$ and inequality
\eqref{I1} turns into $\sum_{i=1}^n\frac{s_i^p}{\mu_i}\geq1$ for all
$(s_1,\cdots,s_n)\in\Omega$, or equivalently
$(\frac{1}{\mu_1},\cdots,\frac{1}{\mu_n})\in D_p(\Omega)$. From
Lemma \ref{lemma3} we deduce that $\mu_i\leq1$ for all
$i=1,\cdots,n$.\\
(ii) It is similar to the proof of Theorem \ref{T1} (ii).\\
(iii) It follows from (i), (ii) and \eqref{L14}.
\end{proof}

Now we want to characterize $G(p)$ for any $p>0$. The next theorem
deals with the case where $p>1$.

%Theorem 2.6%

\begin{theorem}\label{T2}
Let $(X,\|\cdot\|)$ be a normed space and $p>1$. Then the following hold:\\
{\rm(i)} $G(p\ ;k)=\varnothing$, for all $k=0,\cdots,n-2$;\\
{\rm(ii)} $G(p\ ;n-1)\\=\left\{(\mu_1,\cdots,\mu_n)\colon\exists\
j=1,\cdots,n\ ; \mu_j>0,\ \mu_i<0\ (i\neq j)\ and\ \mu_j^\frac{1}{p-1}\geq1+\sum_{i=1,i\neq j}^n|\mu_i|^\frac{1}{p-1}\right\};$\\
{\rm(iii)} $G(p\ ;n)=\mathbb{R}^-\times\cdots\times\mathbb{R}^-$, where $\mathbb{R}^-=\{t\in\mathbb{R}: t <0\};$\\
{\rm(iv)} $G(p)=G(p\ ;n-1)\cup G(p\ ;n)\,.$

\end{theorem}

\begin{proof}
(i) Let $(\mu_1,\cdots,\mu_n)\in G(p\ ;k)$ for some
$k=0,\cdots,n-2$. Thus there exist $j_1,j_2\in\{1,\cdots,n\}$ such
that $\mu_{j_1},\mu_{j_2}>0\ (j_1\neq j_2)$ and also inequality
\eqref{L6.5} holds. Setting $x_i=0\
(i\in\{1,\cdots,n\}\setminus\{j_1,j_2\})$ in inequality \eqref{L6.5}
we have
\begin{eqnarray}\label{L15}
\|x_{j_1}+x_{j_2}\|^p\geq\frac{\|x_{j_1}\|^p}{\mu_{j_1}}+\frac{\|x_{j_2}\|^p}{\mu_{j_2}}\,.
\end{eqnarray}
Also we can consider $x_{j_1},x_{j_2}\in X$ for which
$x_{j_1}+x_{j_2}=0$ and $(x_{j_1},x_{j_2})\neq(0,0)$. Thus
inequality \eqref{L15} implies that
$\frac{\|x_{j_1}\|^p}{\mu_{j_1}}+\frac{\|x_{j_2}\|^p}{\mu_{j_2}}\leq0$, which is a contradiction.\\
(ii) Let $(\mu_1,\cdots,\mu_n)\in\mathbb{R}^n$ satisfy $\mu_j>0$ for
some $j=1,\cdots,n$ and $\mu_i<0$ for all
$i\in\{1,\cdots,n\}\setminus\{j\}$\,.\\
We observe that inequality \eqref{L6.5} holds if and only if
\begin{eqnarray}\label{L9}
\sum_{i=1,i\neq
j}^n\frac{\mu_j\|x_i\|^p}{|\mu_i|\|x_j\|^p}+\frac{\mu_j\|x_1+\cdots+x_n\|^p}{\|x_j\|^p}\geq1
\end{eqnarray}
for all $x_1,\cdots,x_n\in X, x_j\neq0$.\\
Putting
 $\Omega=\left\{\left(\frac{\|x_1\|}{\|x_j\|},\cdots,\frac{\|x_{j-1}\|}{\|x_j\|},\frac{\|x_{j+1}\|}{\|x_j\|},\cdots,\frac{\|x_n\|}{\|x_j\|},\frac{\|x_1+\cdots+x_n\|}{\|x_j\|}\right)\in\mathbb{R}^n\colon
x_1,\cdots,x_n\in X, x_j\neq0\right\}$\\we have
$\Omega\subseteq\big\{(s_1,\cdots,s_n)\colon s_1,\cdots,s_n\geq0,\
\sum_{i=1}^ns_i\geq1\big\}$, $S\subseteq\Omega\subseteq\overline{\Omega}$ and
inequality \eqref{L9} turns into
$$\frac{\mu_j}{|\mu_1|}s_1^p+\cdots+\frac{\mu_j}{|\mu_{j-1}|}s_{j-1}^p+\frac{\mu_j}{|\mu_{j+1}|}s_j^p+\cdots+\frac{\mu_j}{|\mu_n|}s_{n-1}^p+\mu_js_n^p\geq1$$
for all $(s_1,\cdots,s_n)\in\Omega$, or equivalently
\begin{eqnarray}\label{L10}
\left(\frac{\mu_j}{|\mu_{1}|},\cdots,\frac{\mu_j}{|\mu_{j-1}|},\frac{\mu_j}{|\mu_{j+1}|},\cdots,\frac{\mu_j}{|\mu_{n}|},\mu_j\right)\in
D_p(\Omega)\,.
\end{eqnarray}
From Lemma \ref{lemma2} and \eqref{L10} we deduce that
$$\mu_j\geq
h_p\left(\frac{\mu_j}{|\mu_{1}|},\cdots,\frac{\mu_j}{|\mu_{j-1}|},\frac{\mu_j}{|\mu_{j+1}|},\cdots,\frac{\mu_j}{|\mu_{n}|}\right)\,,$$
that is
\begin{eqnarray}\label{L11}
\mu_j\geq\left(1-\sum_{i=1,i\neq
j}^n\frac{|\mu_i|^\frac{1}{p-1}}{\mu_j^\frac{1}{p-1}}\right)^{1-p}\,.
\end{eqnarray}
By a straightforward calculation the following inequality follows
from \eqref{L11}.
$$\mu_j^\frac{1}{p-1}\geq1+\sum_{i=1,i\neq
j}^n|\mu_i|^\frac{1}{p-1}\,.$$ (iii) It is trivial.\\
(iv) It easily follows from (i), (ii), (iii) and \eqref{L14.5}.
\end{proof}

%Theorem 2.7%

\begin{theorem}\label{T4}
Let $(X,\|\cdot\|)$ be a normed space and $0<p\leq1$. Then the following
hold.\\
{\rm(i)} $G(p\ ;k)=\varnothing$, for all $k=0,\cdots,n-2$;\\
{\rm(ii)} $G(p\ ;n-1)=\\
\left\{(\mu_1,\cdots,\mu_n)\colon\exists\ j=1,\cdots,n;\ \mu_j>0,\ \mu_i<0\ (i\neq j)\ and\ \mu_j\geq \displaystyle{\max_{i\in\{1,\cdots,n\}\setminus\{j\}}}\{1,|\mu_i|\}\right\};$\\
{\rm(iii)} $G(p\ ;n)=\mathbb{R}^-\times\cdots\times\mathbb{R}^-;$\\
{\rm(iv)} $G(p)=G(p\ ;n-1)\cup G(p\ ;n)\,.$
\end{theorem}

\begin{proof}
(i) It is similar to the proof of Theorem \ref{T2} (i).\\
(ii) Let $(\mu_1,\cdots,\mu_n)\in\mathbb{R}^n$ satisfying $\mu_j>0$
for some $j=1,\cdots,n$ and $\mu_i<0$ for all
$i\in\{1,\cdots,n\}\setminus\{j\}$.\\
Putting
$\Omega=\left\{\left(\frac{\|x_1\|}{\|x_j\|},\cdots,\frac{\|x_{j-1}\|}{\|x_j\|},\frac{\|x_{j+1}\|}{\|x_j\|},\cdots,\frac{\|x_n\|}{\|x_j\|},\frac{\|x_1+\cdots+x_n\|}{\|x_j\|}\right)\in\mathbb{R}^n\colon
x_1,\cdots,x_n\in X, x_j\neq0\right\}$\\we have
$\Omega\subseteq\left\{(s_1,\cdots,s_n)\colon s_1,\cdots,s_n\geq0,\
\sum_{i=1}^ns_i\geq1\right\}$ and
$\{e_1,\cdots,e_n\}\subseteq\Omega\subseteq\overline{\Omega}$. Passing the proof of
Theorem \ref{T2} (ii) we observe that
$\left\|\sum_{i=1}^nx_i\right\|^p\geq\sum_{i=1}^{n}\frac{\|x_i\|^p}{\mu_i}$ for
all $x_1,\cdots,x_n\in X$ if and only if
$$\left(\frac{\mu_j}{|\mu_{1}|},\cdots,\frac{\mu_j}{|\mu_{j-1}|},\frac{\mu_j}{|\mu_{j+1}|},\cdots,\frac{\mu_j}{|\mu_{n}|},\mu_j\right)\in
D_p(\Omega)\,.$$
From Lemma \ref{lemma3} we deduce that
$$\mu_j\geq1\ \ \  {\rm and}\ \ \mu_j\geq|\mu_i|\ \ (i\in\{1,\cdots,n\}\setminus\{j\}).$$
or equivalently we get $\mu_j\geq
\displaystyle{\max_{i\in\{1,\cdots,n\}\setminus\{j\}}}\{1,|\mu_i|\}$.\\
(iii) It is trivial.\\
(iv) It follows from (i), (ii), (iii) and \eqref{L14.5}.
\end{proof}

%Corollary 2.8%

\begin{corollary}
Suppose that $(X,\|\cdot\|)$ is a normed space and $p>0$. Let
$$H(p):=\left\{(\mu_1,\cdots,\mu_n)\in\mathbb{R}^{n}\colon\ \left\|\sum_{i=1}^nx_i\right\|^p
\leq\left|\sum_{i=1}^{n}\frac{\|x_i\|^p}{\mu_i}\right|\ {\rm for\
all}\ x_1,\cdots,x_n\in X\right\}\,.$$ Then the following assertions hold:\\
{\rm(i)} If $p>1$, then
\begin{eqnarray*}
H(p)&=&F(p)\cap-F(P)\\&=&\left\{(\mu_1,\cdots,\mu_n)\colon \ \mbox{either}\ \mu_1,\cdots,\mu_n>0 \ \mbox{or}\
\mu_1,\cdots,\mu_n<0\ \mbox{as well as}\
\sum_{i=1}^n|\mu_i|^{\frac{1}{p-1}}\leq 1\right\}\,;
\end{eqnarray*}
{\rm(ii)} If $0<p\leq1$, then $H(p)=(0,1]^n\cup[-1,0)^n$.
\end{corollary}
\begin{proof}
One can easily observe that
$$H(p)=(F(p)\cap-G(p))\cup(-F(p)\cap G(p)).$$
Now (i) follows from Theorems \ref{T1} and \ref{T2} and (ii) follows
from Theorems \ref{T3} and \ref{T4}.
\end{proof}

The special cases of our results in the case where $n=2$ and $p \geq 1$ give rise to
some of main results of Takahasi et al \cite [Theorems 1.1 and
4.1]{Takahasi-Ras-Sai-Takahashi}.

\bibliographystyle{amsplain}

\end{document}